\documentclass[11pt]{amsart}
\usepackage{bbm,hyperref}
\usepackage{amsmath}  
\usepackage{amsthm}     
\usepackage{enumerate} 
\usepackage{physics}
\usepackage{mathtools}
\usepackage{bigints}
\usepackage[margin=1.0in]{geometry} 
\usepackage{xcolor}
\usepackage{soul}

\DeclareMathOperator{\sign}{sign}

\newcommand{\lemref}{Lemma~\ref}

\newcommand{\figref}{Figure~\ref}

\newtheorem{remark}{Remark}
\newtheorem{prop}{Proposition}

\newtheorem{lemma}{Lemma}

\newtheorem{definition}{Definition}

\numberwithin{remark}{section}
\numberwithin{prop}{section}
\numberwithin{thm}{section}
\numberwithin{lemma}{section}
\numberwithin{definition}{section}
\numberwithin{equation}{section}

\title{On non-uniqueness in mean field games}
\author[]{Erhan Bayraktar} \thanks{This research was supported in part by the National Science Foundation under grants DMS-1613170.}  
\address{Department of Mathematics, University of Michigan}
\email{erhan@umich.edu}
\author[]{Xin Zhang} 
\address{Department of Mathematics, University of Michigan}
\email{zxmars@umich.edu}

\date{\today}
\begin{document}
\maketitle

\begin{abstract}
We analyze an $N+1$-player game and the corresponding mean field game with state space $\{0,1\}$. The transition rate of $j$-th player is the sum of his control $\alpha^j$ plus a minimum jumping rate $\eta$. Instead of working under monotonicity conditions, here we consider an anti-monotone running cost. We show that the mean field game equation may have multiple solutions if $\eta < \frac{1}{2}$. We also prove that that although multiple solutions exist, only the one coming from the entropy solution is charged (when $\eta=0$), and therefore resolve a conjecture of \cite{2019arXiv190305788H}. 
\end{abstract}

\section{Introduction}
The theory of mean field games (MFGs) was introduced recently (2006-2007) independently by Lasry, Lions (see  \cite{MR2269875}, \cite{LASRY2006679}, \cite{MR2295621}) and Caines, Huang, Malham\'{e} (see  \cite{4303232}, \cite{MR2346927}). It is an analysis of 
limit models for symmetric weakly interacting $N+1$-player differential games (see e.g. \cite{MR3752669}, \cite{MR3753660}). The solution of MFGs provides an approximated Nash Equilibrium. It also under some conditions follows that MFGs are limit points of $N+1$-player Nash equilibria. 

The influential work \cite{MR3967062} by Cardaliaguet, Delarue, Lasry, and Lions established the convergence of closed loop equilibria using the the so-called master equation, which is a partial differential equation with terminal conditions whose variable are time, state and measure. It is known that under the monotonicity condition, the master equation possess a unique solution, which is used to show the above convergence. A similar analysis was carried in finite state mean field games by Bayraktar and  Cohen \cite{MR3860894} and  Cecchin and Pelino \cite{MR4013871} independently obtain the above convergence result (as well as the the analysis of its fluctuations).

In this paper, we consider a case when the monotonicity assumption is not satisfied and resolve a conjecture of \cite{2019arXiv190305788H}, in which a two-state mean field game with Markov feedback strategies is analyzed. In this game the transition rate of each player is the sum of his control and a background jump rate $\eta \geq 0$. Supposing an anti-monotone running cost (follow the crowd game), \cite{2019arXiv190305788H} poses a conjecture on the nature of the limits of $N+1$-player Nash equilibrium. We proceed by using similar techniques to \cite{MR3981375}, which considers an anti-monotone terminal condition. In particular, we again rely on the entropy solution of the master equation to prove the convergence and show that the limit of $N+1$-player Nash equilibrium selects the unique mean field equilibrium induced by this entropy solution. 
In \cite{MR3981375}, they showed that the mean field game equation has at most three equations, while in our model if $\eta<\frac{1}{2}$, the number of solutions is increasing with time horizon and can be arbitrarily large. Also, the entropy solution in our case cannot be written down explicitly, and so we need to construct using the characteristics and check that it is entropic. For numerical methods towards the convergence of $N+1$ player games to entropy solution, we refer readers to the work of Gomes et al. \cite{MR3268060}. Let us mention the recent work by \cite{MR4046528}, where they study linear-quadratic mean field games in the diffusion setting. To re-establish the uniqueness of MFG solutions, they add a common noise and prove that the limit of MFG solutions as noise tends to zero is just the solution induced by the entropy solution of the master equation without common noise. 

The paper is organized as follows. In Section~\ref{sec:model}, we introduce the $N+1$-player game we are considering, and introduce the equations characterizing the mean field equilibria. In Section~\ref{sec:non-uni}, we show that the forward backward equation characterizing the mean field game possesses a unique solution if $\eta \geq \frac{1}{2}$, may have multiple solutions if $\eta < \frac{1}{2}$. Furthermore, we also determine the number of solutions.  In Section~\ref{eq:ME}, we explicitly find the entropy solution of the master equation. In Section~\ref{sec:conv}, we show that if $\eta =0$ each player in the $N+1$-player game will follow the majority and briefly present that the optimal trajectories of $N+1$-player game converges to the optimal trajectory induced by the entropy solution of the master equation.

\section{Two states mean field games}\label{sec:model}
We consider the $N+1$-players game with state space $\Sigma=\{0,1\}$, and denote the state of players by $\mathbf{Z}(t):=(Z_j(t))_{j=1}^{N+1}$, which evolves as controlled Markov processes. The jump rate of $Z_j(t)$ is given by $\alpha^j(t,\mathbf{Z}(t))+ \eta$, where $\alpha^j:[0,T] \times \Sigma^{N+1} \to [0,+\infty)$ is the control of player $j$ and $\eta \geq 0$ is the minimum jump rate, i.e., 
$$\mathbb{P}[Z_j(t+h)=1-i| Z_j(t)=i]=(\alpha^j(t, \mathbf{Z}(t)) + \eta ) h + o(h) .$$
Denote by $\mathcal{A}$ the collection of all the measurable and locally integrable functions $[0,T] \times \Sigma^{N+1} \to [0, +\infty)$, and by $\boldsymbol{\alpha}^{N+1}=(\alpha^1, \dotso, \alpha^{N+1})  \in \mathcal{A}^{N+1}$ the control of all players. It is can be easily seen that the law of Markov process is determined by the control vector $\boldsymbol{\alpha}^{N+1} $.

Let the empirical measure of player $j$ at time $t$ to be $$\theta^{N+1,j}(t)= \frac{1}{N} \sum\limits_{k=1, k \not = j}^{N+1} \delta_{Z_k(t)=0}.$$
Then given the running cost function 
\begin{equation}
f(i,\theta)=|1-\theta-i|=
\begin{cases}
1- \theta       & i=0  \\
\theta      & i=1 , \\
\end{cases}
\end{equation}
the control vector $\boldsymbol{\alpha}^{N+1} \in \mathcal{A}^{N+1}$ and it is associated Markov process $(\mathbf{Z}(t))_{0 \leq t \leq T}$, the objective function of the $k$-th player is defined by 
$$J^{N+1}_k(\boldsymbol{\alpha}^{N+1})= \mathbb{E}\bigg[\int_0^T f(Z_k(t),\theta^{N+1,k}(t))+ \frac{ \alpha^k(t, \mathbf{Z}(t))}{2} dt        \bigg]$$
For a control vector  $\boldsymbol{\alpha}^{N+1} \in \mathcal{A}^{N+1}$ and $\beta \in \mathcal{A}$, define the perturbed control vector by 
\begin{equation*}
[\boldsymbol{\alpha}^{N+1, -j}; \beta]_k:=
\begin{cases}
\alpha_k,   & k \not = j  \\
\beta, & k=j.   \\
\end{cases}
\end{equation*}

\begin{definition}
A control vector $\boldsymbol{\alpha}^{N+1} \in \mathcal{A}^{N+1}$ is a Nash Equilibrium if for any $k=1, \dotso, N+1$ 
$$J_k^{N+1}(\boldsymbol{\alpha}^{N+1})= \inf\limits_{\beta \in \mathcal{A}} J_k^{N+1}([\boldsymbol{\alpha}^{N+1, -}; \beta]      ).$$
\end{definition}

To find the Nash equilibrium, it is standard to solve its corresponding Hamilton-Jacobi equations for value functions $V^{N+1}(t,i,\theta), i=0,1$   (see e.g. \cite{MR3072242}).
\begin{equation}\label{HJB1}
\begin{cases}
 -\frac{d}{dt} V^{N+1}(t,i,\theta) = f(i,\theta)- \frac{(\alpha^{N+1}_*(t,i,\theta))^2 }{2}\\
 \ \ \ \ \ \ \ \ \ \ \ \ \ \ \ \ + \eta (V^{N+1}(t,1-i,\theta)-V^{N+1}(t,i, \theta))   \\ 
 \ \ \ \ \ \ \ \ \ \ \ \ \ \ \ \ +N(1-\theta)\bigg(\alpha^{N+1}_*(t,1,\theta+\frac{1-i}{N}) +\eta \bigg)(V^{N+1}(t,1, \theta +\frac{1}{N})-V^{N+1}(t,1,\theta))\\
 \ \ \ \ \ \ \ \ \ \ \ \ \ \ \ \ +N\theta \bigg( \alpha^{N+1}_*(t,0,\theta-\frac{i}{N}) +\eta \bigg)(V^{N+1}(t,1,\theta-\frac{1}{N})-V^{N+1}(t,1,\theta)),   \tag{HJB} \\ 
 V^{N+1}(T,i,\theta)= 0, \\
\end{cases}
\end{equation}
where the optimal control  is given by
$$a^{N+1}_*(t, i, \theta)=(V^{N+1}(t,i,\theta)-V^{N+1}(t, 1-i, \theta))_+.$$

It is also easy to write down the corresponding mean field game equation,
\begin{equation}\label{MFG}
\begin{cases}
\frac{d}{dt} \theta(t)=(1-\theta(t))( (u(t,1)-u(t,0)   )_+ + \eta      )  -\theta(t)( (u(t,0)-u(t,1)   )_+ + \eta      ),   \tag{MFG} \\
 -\frac{d}{dt} u(t,i)=f(i,\theta)-\eta(u(t,i)-u(t,1-i)        )-\frac{ ( (u(t,i)-u(t,1-i)   )_+)^2         }{2},  \\
 \theta(0)=\bar{\theta},  \\
 u(T,i)=0,    \\
\end{cases}
\end{equation}
and  see e.g. \cite{MR3072242} and the corresponding master equation, 
 the corresponding master equation, 
\begin{equation}\label{ME}
\begin{cases}
-\frac{\partial }{\partial t}U(t,i,\theta)= f(i,\theta)-\frac{[(U(t,i,\theta)-U(t,1-i, \theta)_+]^2}{2}+\eta (U(t,1-i, \theta)-U(t,i,\theta)) \\
\ \ \ \ \ \ \ \ \ \ \ \ \ \ \ \ \ +\frac{\partial }{\partial \theta}U(t,i,\theta)((U(t,1,\theta)-U(t,0,\theta)_+ +\eta)(1-\theta) \\
\ \ \ \ \ \ \ \ \ \ \ \ \ \ \ \ \ -\frac{\partial }{\partial \theta}U (t,i,\theta)((U(t,0,\theta)-U(t,1,\theta)_+ + \eta)\theta, \tag{ME} \\
U(T,i,\theta)=0,
\end{cases}
\end{equation}
see Bayraktar, Cohen \cite{MR3860894} and  Cecchin, Pelino \cite{MR4013871}.
Recall from the latter two references that the uniqueness of \eqref{MFG} and \eqref{ME} is guaranteed by the so-called monotonicity condition, i.e., for every $\theta, \theta^{'} \in [0,1]$, 
$$\sum\limits_{i=0,1} (-1)^i (f(i,\theta)-f(i,\theta^{'}))(\theta-\theta^{'}) \geq 0,$$
which does not hold true with our choice of running cost.

\section{non-uniqueness}\label{sec:non-uni}
We show that the mean field equations \eqref{MFG} may have multiple solutions. Taking $$y(t)=u(t,1)-u(t,0), \quad x(t)=2\theta(t)-1,$$
then $\eqref{MFG}$ becomes

\begin{equation}\label{eq1}
\begin{cases}
 \frac{d}{dt} x =y-x|y|-2 \eta x   \\
 -\frac{d}{dt}  y =x-\frac{1}{2}y|y|-2 \eta y    \\
 y(T)=0, x(0)=2\bar{\theta}-1.
\end{cases}
\end{equation}
The second one of  \eqref{eq1} is equivalent to 
\begin{equation}\label{eq2}
\begin{aligned}
x=\frac{1}{2}y|y|+2 \eta y-\frac{d}{dt}  y.
\end{aligned}
\end{equation}
Taking derivative with respect to $t$ in \eqref{eq2} and in conjunction with \eqref{eq1}, we obtain 
\begin{equation}\label{eq3}
\begin{aligned}
\frac{d^2}{dt^2}  y + y-\frac{1}{2}y^3-3 \eta |y|y-4{\eta}^2y=0.
\end{aligned}
\end{equation}

For simplicity, we time reverse the system and try to solve
\begin{equation}\label{eq4}
\begin{cases}
\frac{d^2}{dt^2}y + y-\frac{1}{2}y^3-3 \eta |y|y-4{\eta}^2y=0 \\
\frac{1}{2}y(T)|y(T)|+2 \eta y(T)+\frac{d}{dt} y(T)=x(T)=2\bar{\theta}-1\\
y(0) =0. \\
\end{cases}
\end{equation}
Since \eqref{eq4} contains only the $y$ variable, it can be uniquely solved if imposing the initial conditions $y(0)=0, \frac{d}{dt} y(0)=v$, and we denote its $\mathcal{C}^1$ solution as $y_v(.)$. Therefore the number of solutions to \eqref{eq4} is just the number of initial velocity $v$ such that $2\bar{\theta}-1=x_v(T)$, where for any $t \geq 0$
\begin{equation}\label{x_v} 
x_v(t):=\frac{1}{2}y_v(t)|y_v(t)|+2 \eta y_v(T)+\frac{d}{dt} y_v(t) 
\end{equation}

We rewrite the differential equation as a derivative with respect to $y$ instead of $t$, i.e., 
$$\frac{d^2 y}{dt^2}=\frac{d}{dt}\bigg(\frac{1}{2}(\frac{dy}{dt})^2\bigg)\frac{dt}{dy}=\frac{d}{dy}\bigg(\frac{1}{2}(\frac{dt}{dy})^{-2}\bigg).$$
We can therefore get an implicit solution 
\begin{equation}\label{solution}
\frac{dt}{dy}=\pm \frac{1}{\sqrt{G(y)+v^2}},
\end{equation}
where $G(y)=\frac{1}{4}y^4+2\eta|y|^3+4 {\eta}^2y^2-y^2.$ 

When $y \geq 0$, the first order derivative of $G$ is 
$$G^{'}(y)=y^3+6\eta y^2+8{\eta}^2y-2y=y(y+3\eta-\sqrt{{\eta}^2+2})(y+3 \eta + \sqrt{\eta^2+2}).$$
It is then easy to conclude the following results 
\begin{itemize}
\item If $\eta \geq \frac{1}{2}$, the function $G(y)$ is strictly increasing for $y \geq 0$;
\item If $0 \leq \eta < \frac{1}{2}$, the function $G(y)$ decreases on the interval $[0, \sqrt{{\eta}^2+2}-3\eta]$ and increases on the interval $ [\sqrt{{\eta}^2+2}-3\eta, +\infty)$;
\item If $\eta < \frac{1}{2}, |v| < v_0$, the function $G(y)+v^2$ maybe negative for some $y \in \mathbb{R}$. Let us denote by $y(v)$ the smallest positive root of $G(y)+v^2=0$. Since the function $y \mapsto G(y)$ first decreases to $-v_0^2$ over the interval $[0, \sqrt{{\eta}^2+2}-3\eta]$, and then increasing to $+\infty$ over the interval $[\sqrt{{\eta}^2+2}-3\eta, +\infty)$, we know that the function $y \mapsto G(y)+v^2$ decreases over $[0, y(v))$ and crosses $0$ at $y(v)$, which implies that $y(v)$ is a simple root.
\end{itemize}
Let $v_0:=\sqrt{-G(\sqrt{{\eta}^2+2}-3\eta)}$ if $\eta < \frac{1}{2}$.
and 
\begin{equation}\label{eq:defT}
T(v):=\int_0^{y(v)} \frac{dz}{\sqrt{G(z)+v^2}}, \quad v \in (0, v_0),
\end{equation}
whose role will be clear in the next result.


\begin{lemma}\label{y_v}
The following properties hold for solutions $y_v(.)$,
\begin{itemize}
\item  $y_v(.)$ is strictly increasing if $v >0$, strictly decreasing if $v<0$, identically $0$ if $v=0$;
\item If either $\eta \geq \frac{1}{2}, v \in \mathbb{R}$ or $\eta < \frac{1}{2}, |v| \geq v_0$, then the solution $y_v(t)< +\infty$ if and only if $ t<\int_0^{+\infty} \frac{dz}{\sqrt{G(z)+v^2}}$.  Furthermore, $y_v(.)$ is strictly increasing if $v >0$, strictly decreasing if $v<0$;
\item If $\eta < \frac{1}{2}, |v| \in (0, v_0)$, the solution $y_v(.)$ is a periodic function.
\end{itemize}
\end{lemma}
\begin{proof}
The first statement is clear. We prove the rest by writing down the unique $\mathcal{C}^1$ solution explicitly.

If either $\eta \geq \frac{1}{2}, v \in \mathbb{R}$ or $\eta < \frac{1}{2}, |v| \geq v_0$, then $G(z)+v^2  \geq 0$ for any $z \in \mathbb{R}$ and thus we obtain from \eqref{solution} that
\begin{equation*}
t= \sign(v) \int_0^y \frac{dz}{\sqrt{G(z)+v^2} }.
\end{equation*}
Since the function $y \mapsto \int_0^y \frac{dz}{\sqrt{G(z)+v^2} }$ is strictly increasing, for any $t <  \int_0^{+\infty} \frac{dz}{\sqrt{G(z)+v^2} }$, we can find a unique $y_v(t)$ such that $$t=\int_0^{y_v(t)}  \frac{dz}{\sqrt{G(z)+v^2} }.$$
It can be seen that the function $t \mapsto y_v(t)$ is $\mathcal{C}^1$, and therefore is the unique solution to \eqref{eq4}.

 Since $G(y_v(t))+v^2$ is always nonnegative, the solution $y_v(t)$ must oscillate between $[-y(v), y(v)]$. For any $0 \leq t \leq T(v)$, there exists a unique $y_v(t)$ such that $$t=\int_0^{y_v(t)}  \frac{dz}{\sqrt{G(z)+v^2} }.$$ Define a periodic function, still denoted by $y_v(.)$,
\[
y_v(t)=
\begin{cases}
 y_v(t-4kT(v)) & \ t \in [4kT(v), (4k+1)T(v)), \\
y_v((4k+2)T(v)-t) & \  t \in [(4k+1)T(v), (4k+2)T(v)), \\
-y_v(t-(4k+2)T(v)) & \ t \in [(4k+2)T(v), (4k+3)T(v)), \\
-y_v((4k+4)T(v)-t) & \ t \in [(4k+3)T(v), (4k+4)T(v)).  \\
\end{cases}
\]
It can be easily seen that $y_v(t)$ is the unique $\mathcal{C}^1$ solution to \eqref{eq4}.
\end{proof}

\begin{prop}
If $\eta \geq \frac{1}{2}$, then $x_v(T)$ is strictly increasing with respect to $v$ and therefore \eqref{eq4} has unique solution. 
\end{prop}
\begin{proof}
It can be seen that both of the equation \eqref{eq4} and the function $v \mapsto x_v(T)$ are odd. Therefore $y_{-v}(.)=-y_v(.)$, $x_{-v}(T)=-x_v(T)$, and we only need to prove the proposition for $v \geq 0$.  

The strictly decreasing function $v \mapsto \int_0^{+\infty} \frac{dz}{\sqrt{G(z)+v^2}} $ approaches $+\infty$ as $v \to 0$, approaches $0$ as $v \to +\infty$.
Therefore any positive $T$ there exists a unique $u >0$ such that $$
\int_0^{+\infty} \frac{dz}{\sqrt{G(z)+u^2}}=T.$$ As a result of Lemma~\ref{y_v}, the solution $y_v(.)$ is finite at $T$ if and only if $v<u$, and there exists a unique $y_v(T) >0$ such that $$T=\int_0^{y_v(T)} \frac{dz}{\sqrt{G(z)+v^2}},$$
and also $\frac{dy_v}{dt}|_T=\sqrt{G(y_v(T))+v^2}.$
Suppose $0 \leq v_1 < v_2<u$. Due to the fact that $G(z)+v_1^2 < G(z)+v_2^2, \forall z \in \mathbb{R}$, we obtain $$y_{v_1}(T)<y_{v_2}(T), \frac{d}{dt} y_{v_1}(T)< \frac{d}{dt} y_{v_2}(T),$$
from which we can conclude $x_{v_1}(T)<x_{v_2}(T)$. As a result of $\lim\limits_{v \to u} y_v(T)=+\infty$,  we obtain $\lim\limits_{v \to u} x_v(T)=+\infty$, and thus there exists a unique solution to \eqref{eq4} for any $2\bar{\theta}-1 \in \mathbb{R}$.

\end{proof}

As a result of the above proposition, the mean field equation \eqref{eq1} may have multiple solutions only if $\eta< \frac{1}{2}$. To find the number of solutions, we study the period of $y_v(.)$ in the following lemma. Note that since $y_{-v}(t)=-y_v(t)$ and $y_0(t)=0$, it suffices for us to consider the period of $y_v(.)$ for $v \in (0, v_0)$.

\begin{lemma}\label{monotone}
Suppose $0 \leq \eta <\frac{1}{2}$, $v \in (0, v_0)$, and $y(v)$ is the smallest postive root of $z \mapsto G(z)+v^2$. Recall \eqref{eq:defT} and define
\[
H(v):= \int_v^{y(v)} \frac{dz}{\sqrt{G(z)+v^2}}.\]
 Take $T(v)=T(-v), H(v)=H(-v)$ if $v \in (-v_0, 0)$. 
Then both $T(.)$ and $H(.)$ are increasing with respect to $v$ over the interval $(0,v_0)$, and $\lim\limits_{v \to v_0}T(v)=+\infty .$
\end{lemma}
\begin{proof}
By the definition, we have $G(y)+v^2=(\frac{y^2}{2}+2\eta |y|)^2+v^2-y^2$, from which we can conclude that $y(v) \geq v$, and therefore $H(v)$ is positive. 

By change of variable $p=\frac{z}{y(v)}$, we obtain 
\begin{equation*}
T(v)=\bigintss_0^1 \frac{dp}{\sqrt{\frac{G(y(v)p)}{y(v)^2}+\frac{v^2}{y(v)^2} }}= \bigintss_0^1 \frac{dp}{\sqrt{ \frac{1}{4}y(v)^2p^4+2\eta y(v)p^3+(4{\eta}^2-1)p^2+\frac{v^2}{y(v)^2} }}.
\end{equation*}
Denote the square of the bottom of the integrand by $P(v,p)$, i.e., 
$$P(v,p):= \frac{1}{4}y(v)^2p^4+2\eta y(v)p^3+(4{\eta}^2-1)p^2+\frac{v^2}{y(v)^2}.$$
To prove $T(v)$ is increasing, it suffices to show that $P(v,p)$ is decreasing with respect to $v$ for any fixed $p \in [0,1]$.

Since $y(v)$ is an increasing function of $v$, the derivative $\frac{dP}{dv}(v,p)$ is no larger than $\frac{dP}{dv}(v,1)$, which is equal to $0$ according to the definition of $y(v)$, 
$$\frac{dP}{dv}(v,1)=\frac{d (G(y(v))+v^2)}{dv}=0.$$
Therefore $P(v_1,p) \geq P(v_2,p)$ for any $p \in [0,1], 0<v_1< v_2 < v_0$. 

We can also rewrite $H(v)$ as 
\begin{equation*}
H(v)=\int_{\frac{v}{y(v)}}^1 \frac{dp}{\sqrt{P(v,p)}},
\end{equation*}
and it is enough to show that $v \mapsto \frac{v}{y(v)}$ is decreasing. Taking derivative of the following equation with respect to $v$, 
$$G(y(v))+v^2=0,$$
we get $\frac{dy(v)}{dv}=-\frac{2v}{G^{'}(y(v))},$ and thus $$\frac{d}{dv}( \frac{v}{y(v)}     )=\frac{y(v)-v\frac{dy(v)}{dv}}{y(v)^2}=\frac{y(v)+\frac{2v^2}{G^{'}(y(v))}}{y(v)^2}.$$
As a result of $\frac{dy(v)}{dv} \geq 0$, we obtain that $G^{'}(y(v))<0$ and $\frac{d}{dv}( \frac{v}{y(v)} ) \leq 0$ is equivalent to $G^{'}(y(v))y(v)+2v^2 \geq 0.$ We conclude our claim by the following computation,
\begin{equation*}
\begin{aligned}
G^{'}(y(v))y(v)+2v^2 & =G^{'}(y(v))y(v)+2v^2-2(G(y(v))+v^2) \\
&=\frac{1}{2}y(v)^4+2 \eta y(v)^3>0  \\
\end{aligned}
\end{equation*}

In the end, it can be seen that the function $z \mapsto G(z)+v_0^2$  
is always positive over the interval $[0, +\infty)$ and only attains $0$ at $z=\sqrt{{\eta}^2+2}-3\eta$. Since $G(z)+v_0^2$ is a polynomial, we obtain that $y(v_0)=\sqrt{{\eta}^2+2}-3\eta$, $(z-\sqrt{{\eta}^2+2}+3\eta)^2$ is a factor of $G(z)+v_0^2$, and hence $$\lim\limits_{v \to v_0}T(v)= \int_0^{\sqrt{{\eta}^2+2}-3\eta} \frac{dz}{\sqrt{G(z)+v_0^2}}=+\infty.$$
\end{proof}

For each $k \in \mathbb{N}$, define $T_k(v):=(2k-1)T(v)+H(v)$ if $|v| \in (0,v_0)$, and $T_k(v):=+\infty$ if $|v|>v_0$. Now we show that for $v \not =0$, $\{T_k(v): k \in \mathbb{N}\}$ is the set of times $T$ such that $x_v(T)$ attains $0$ ($T_k(v)=+\infty$ for $|v| \geq v_0$ simply implies that $x_v(t)$ never reaches $0$ for those $v$). As a result of Lemma~\ref{y_v}, the function $x_v(T)$ can equal to $0$ only if $\eta < \frac{1}{2}, |v| \in (0, v_0)$ or $v=0$. Setting $x_v(T)=0$, by \eqref{x_v} we get 
\begin{equation*}
\begin{aligned}
0=x_v(T) &=\frac{1}{2}y_v(T)|y_v(T)|+2 \eta y_v(T)+\frac{d}{dt} y_v(T) \\
&=\frac{1}{2}y_v(T)|y_v(T)|+2 \eta y_v(T)+\sign(\frac{d}{dt} y_v(T)) \sqrt{G(y_v(T))+v^2 }.
\end{aligned}
\end{equation*}
Moving the last term to the left, taking square of both sides and plugging in the formula of $G(y)$, it becomes 
$$(\frac{1}{2}y_v(T)|y_v(T)|+2 \eta y_v(T))^2+v^2- (y_v(T))^2=(\frac{1}{2}y_v(T)|y_v(T)|+2 \eta y_v(T))^2,$$
which is equivalent to $v^2- (y_v(T))^2=0$. Therefore we obtain  that $|y_v(T)|=v, \sign(y_v(T))=- \sign(\frac{d}{dt} y_v(T))$, from which we conclude that $x_v(T)=0$ if and only if $T=T_k(v)$ or $v=0$. 

Therefore $T_1(v)$ is the first time $x_v(t)$ reaches $0$. Taking $T_k(0+):=\lim\limits_{v \downarrow 0} T_k(v)$, it can be seen that for $t \leq T_1(0+), v \not = 0$, we have $x_v(t) \not =0$. 
Before computing the number of solutions, we still need one more result, which is also important for us to construct the entropy solution of the master equation in the next section.

\begin{lemma}\label{10}
Suppose $\eta < \frac{1}{2}$. Then for any $(x,t) \in \mathbb{R}  \times \mathbb{R}_+ \setminus \{0\} \times \mathbb{R}_+$, there exists a unique $v(x,t) \in \mathbb{R}_+ $ such that $x_v(t)=x, t< T_1(v)$ (simply take $v(x,t)=0$ if $x=0$).
\end{lemma}
\begin{proof}
\textbf{Step 1.} For any $0<v_1<v_2 \leq v_0$, we prove that $y_{v_1}(t) < y_{v_2}(t), \forall t \in (0, T_1(v_1)]$. Otherwise suppose $y_{v_1}(t)=y_{v_2}(t)$ for some $t \in (0, T_1(v_1)]$. If $t \leq T(v_1)$, as in the proof of Lemma~\ref{y_v} we have 
\begin{equation}
t=\int_0^{y_{v_1}(t)} \frac{dz}{\sqrt{G(z)+v_1^2}} =\int_0^{y_{v_2}(t)} \frac{dz}{\sqrt{G(z)+v_2^2}},
\end{equation}
which is impossible since $G(z)+v_1^2<G(z)+v_2^2$. If $t \in (T(v_1), T(v_2)]$, then $y_{v_2}(t)> y_{v_2}(T(v_1))> y_{v_1}(T(v_1))>y_{v_1}(t)$, which is contradictory to our assumption. If $t \in (T(v_2), T_1(v_1)]$, we have 
\begin{equation*}
\begin{aligned}
2T(v_1)-t = \int_0^{y_{v_1}(t)} \frac{dz}{\sqrt{G(z)+v_1^2}} > \int_0^{y_{v_2}(t)} \frac{dz}{\sqrt{G(z)+v_2^2}}= 2T(v_2)-t,
\end{aligned}
\end{equation*}
which contradicts to \lemref{monotone}.

\textbf{Step 2.} For any $v_0 \leq v_1<v_2,   t \in \big(0, \int_0^{+\infty} \frac{dz}{\sqrt{G(z)+v^2}}\big]$, we have $y_{v_1}(t)<y_{v_2}(t)$, which can be proved as in \textbf{Step 1}.

\textbf{Step 3.} For any $0<v_1<v_2 \leq v_0$, we prove that $x_{v_1}(t)<x_{v_2}(t), \forall t \in [0, T_1(v_1)]$. Otherwise suppose $t= \sup \{ t: x_{v_1}(t)=x_{v_2}(t), t \leq T_1(v_1) \} $, where supreme is attained by the continuity of $x_{v_1}(.)$ and $x_{v_2}(.)$. To show the contradiction, we prove that $\frac{d}{dt}(x_{v_2}(t)-x_{v_1}(t)) < 0$, in which case these two curves have to intersect after time $t$ since $x_{v_2}$ decreases to $0$ at time $T_1(v_2) > T_1(v_1)$. 

If $t \geq T(v_1)$, we have 
\begin{equation*}
\begin{aligned}
x_{v_1}(t) & =\frac{1}{2}y_{v_1}(t)^2+2 \eta y_{v_1}(t)- \sqrt{G(y_{v_1}(t))+v_1^2}  \\
&=\frac{1}{2}y_{v_2}(t)^2+2 \eta y_{v_2}(t) + \sign(\frac{d}{dt} y_{v_2}(t))  \sqrt{G(y_{v_2}(t))+v_2^2}= x_{v_2}(t). \\
\end{aligned}
\end{equation*}
Since we proved $y_{v_1}(t)< y_{v_2}(t)$, the derivative $\frac{d}{dt} y_{v_2}(t)$ must be negative, and hence 
\begin{equation}\label{==}
\frac{1}{2}y_{v_1}(t)^2+2 \eta y_{v_1}(t)- \sqrt{G(y_{v_1}(t))+v_1^2}  
=\frac{1}{2}y_{v_2}(t)^2+2 \eta y_{v_2}(t) -  \sqrt{G(y_{v_2}(t))+v_2^2}.
\end{equation}
Combining \eqref{==} and  $\frac{d}{dt} y_{v_i}(t)=-\sqrt{ G(y_{v_i}(t))+v_i^2}, i=1,2$ , we obtain 
\begin{equation*}
\begin{aligned}
\frac{d}{dt}(x_{v_2}(t)-x_{v_1}(t)) = & y_{v_1}(t)\bigg(\sqrt{G(y_{v_1}(t))+v_1^2}-\frac{1}{2}y_{v_1}(t)^2-2 \eta y_{v_1}(t) +1) \bigg) \\
& -y_{v_2}(t)\bigg(\sqrt{G(y_{v_2}(t))+v_2^2}-\frac{1}{2}y_{v_2}(t)^2-2 \eta y_{v_2}(t) +1) \bigg). \\
\end{aligned}
\end{equation*}
Because of \eqref{==} and the fact that $y_{v_2}(t) > y_{v_1}(t)$, we deduce that  $\frac{d}{dt}(x_{v_2}(t)-x_{v_1}(t)) <0$ is equivalent to $\sqrt{G(y_{v_2}(t))+v_2^2}-\frac{1}{2}y_{v_2}(t)^2-2 \eta y_{v_2}(t) +1>0$, which is true since 
\begin{equation*}
\begin{aligned}
\sqrt{G(y_{v_2}(t))+v_2^2}-\frac{1}{2}y_{v_2}(t)^2-2 \eta y_{v_2}(t) +1  & > -\frac{1}{2}y_{v_2}(t)^2-2 \eta y_{v_2}(t) +1 \\
                      &> -\frac{1}{2}(\sqrt{{\eta}^2+2}-3\eta)^2 - 2 \eta (\sqrt{{\eta}^2}+2-3\eta)+1 >0.
\end{aligned}
\end{equation*}

If $t< T(v_1)$, by the same reasoning we have 
$$\frac{1}{2}y_{v_1}(t)^2+2 \eta y_{v_1}(t)+ \sqrt{G(y_{v_1}(t))+v_1^2}  
=\frac{1}{2}y_{v_2}(t)^2+2 \eta y_{v_2}(t) + \sqrt{G(y_{v_2}(t))+v_2^2},$$ 
and also
\begin{equation*}
\begin{aligned}
\frac{d}{dt}(x_{v_2}(t)-x_{v_1}(t)) = & y_{v_2}(t)\bigg(\sqrt{G(y_{v_2}(t))+v_2^2}+\frac{1}{2}y_{v_2}(t)^2+2 \eta y_{v_2}(t) -1 \bigg) \\
& -y_{v_1}(t)\bigg(\sqrt{G(y_{v_1}(t))+v_1^2}+\frac{1}{2}y_{v_1}(t)^2+2 \eta y_{v_1}(t) -1 \bigg). \\
\end{aligned}
\end{equation*}
Accordingly, it suffices to show that $\bigg(\sqrt{G(y_{v_2}(t))+v_2^2}+\frac{1}{2}y_{v_2}(t)^2+2 \eta y_{v_2}(t) -1 \bigg)<0,$
which is equivalent to 
\begin{equation}\label{=}
\sqrt{G(y_{v_2}(t))+v_2^2}<1-\frac{1}{2}y_{v_2}(t)^2-2 \eta y_{v_2}(t).
\end{equation}
Taking square of \eqref{=} , we obtain the equivalent inequality $v_2^2+4 \eta y_{v_2}(t)-1<0.$ Since  $y_{v_2}(t) \leq  y(v_2)$, we conclude our claim by the following computation 
\begin{equation*}
\begin{aligned}
v_2^2+4 \eta y_{v_2}(t)-1 \leq &  v_2^2+4\eta y(v_2)-1  = -G(y(v_2))+4 \eta y(v_2)-1 \\
= & -(\frac{1}{2}y(v_2)^2+2\eta y(v_2)-1)^2 <0.
\end{aligned}
\end{equation*}

\textbf{Step 4.}  For any $v_0 \leq v_1 < v_2, t \in \big(0, \int_0^{+\infty} \frac{dz}{\sqrt{G(z)+v^2}} \big]$, we have $x_{v_1}(t)< x_{v_2}(t)$, which can be proved as in \textbf{Step 3}. 

\textbf{Step 5.} Until now we have shown that the stopped curves $\{x_v(t): 0 \leq t < T_1(v)\}$ do not intersect, and it remains to prove that for any $(x,t) \in \mathbb{R}_+ \times \mathbb{R}_+$, there exists a $v(x,t) \in \mathbb{R}_+$ such that $x_v(t)=x, t < T_1(v)$. 
Note that according to \eqref{eq4}, for any fixed $t$, the couple $(y_v(t), \frac{d}{dt}y_v(t))$ is continuous with respect to the initial velocity $v$, and thus the mapping $v \mapsto x_v(t)$ is also continuous. 

First suppose $x< x_{v_0}(t)$ and $t \leq T_1(0+)$. As a result of  $\lim\limits_{v \to 0} x_v(t)=0, \lim\limits_{v \to v_0} x_v(t)=x_{v_0}(t)$ and the continuity of $v \mapsto x_v(t)$, we know that there must exist some $v \in (0,v_0)$ such that $x_v(t)=x$. The equality $t< T_1(v)$ simply follows from the inequality $t \leq T_1(0+) < T_1(v)$. 

Suppose  $x<x_{v_0}(t)$ and $t > T_1(0+)$.  Since $T_1(v)$ increases to $+\infty$ as $v$ increases to $v_0$, we know that there exists a unique $v' \in (0, v_0)$ such that $t=T_1(v')$, which also implies $x_{v'}(t)=0$.  According to the continuity of $v' \mapsto x_{v'}(t)$, and the fact that $\lim\limits_{v \to v_0} x_v(t)=x_{v_0}(t)$, we know there must exist a $v>v'$ such that $x_{v}(t)=x$, and $t=T_1(v')<T_1(v)$. 

In the end suppose $x>x_{v_0}(t)$. Because the mapping $v \mapsto \int_0^{+\infty} \frac{dz}{\sqrt{G(z)+{v}^2} }$ is decreasing from $+\infty$ to $0$ over the interval $(v_0, +\infty)$, there exists a unique $v' > v_0$ such that $\int_0^{+\infty} \frac{dz}{\sqrt{G(z)+{v'}^2} }=t$, which also implies $x_{v'}(t)=+\infty$. Again by the continuity of $v \mapsto x_v(t)$ and the fact that $\lim\limits_{v \to v_0} x_v(t)=x_{v_0}(t)<x$, there exists a $v>v_0$ such that $x_v(t)=x$.
\end{proof}

\begin{prop}
Suppose $\eta < \frac{1}{2}$. Then there exists a unique solution to \eqref{eq4} for any $T >0$ if $|2 \bar{\theta}-1| \geq 1-{\eta}^2-\eta\sqrt{{\eta}^2+2}$, and the number of solutions to \eqref{eq4} can be arbitrarily large if $|2 \bar{\theta}-1| < 1-{\eta}^2-\eta\sqrt{{\eta}^2+2}$ and  $T$ is large enough. In particular, the number of solutions  with boundary condition $2\bar{\theta}-1=0$ is given by $$1+2 \sup\limits_{k \in \mathbb{N}} \{k : T_k(0+) <T \}.$$
\end{prop}  
\begin{proof}
Recalling $v_0=\sqrt{-G(\sqrt{{\eta}^2+2}-3\eta)}$, we first prove that $x_{v_0}(t)$ is increasing with respect to $t$ and $\lim\limits_{t \to +\infty}x_{v_0}(t)=1-{\eta}^2-\eta\sqrt{{\eta}^2+2}$. 

Taking derivative of the following equation, $$x_{v_0}(t)=\frac{1}{2}y_{v_0}(t)y_{v_0}(t)+2 \eta y_{v_0}(t)+\frac{d}{dt} y_{v_0}(t),$$
we get $\frac{d}{dt} x_{v_0}(t)=(y_{v_0}(t)+2 \eta) \frac{d}{dt} y_{v_0}(t) + \frac{1}{2} G^{'}(y_{v_0}(t)) $.
Therefore $x_{v_0}(t)$ is increasing is equivalent to 
\begin{equation}\label{eq60}
(y_{v_0}(t)+2 \eta) \frac{d}{dt} y_{v_0}(t) \geq  - \frac{1}{2} G^{'}(y_{v_0}(t)).
\end{equation}
Since both sides of \eqref{eq60} are positive, it is enough to show that 
$$(y_{v_0}(t)+2 \eta)^2 (\frac{d}{dt} y_{v_0}(t))^2 -\frac{1}{4} (G^{'}(y_{v_0}(t)))^2>0.$$
Plugging in the equality $\frac{d}{dt} y_{v_0}(t)=\sqrt{G(y_{v_0}(t))+v_0^2  }$ and the formula of $G$, 
the inequality becomes 
\begin{equation*}
\begin{aligned}
2\eta (y_{v_0}(t))^3+(4{\eta}^2-1+v_0^2) (y_{v_0}(t))^2+4\eta v_0^2 y_{v_0}(t)+4 {\eta}^2v_0^2 \geq 0.
\end{aligned}
\end{equation*}
Now we finish proving $x_{v_0}(t)$ is increasing by the following equality,
\begin{equation*}
\begin{aligned}
2\eta (y_{v_0}(t))^3+(4{\eta}^2-1+v_0^2) (y_{v_0}(t))^2+4\eta v_0^2 y_{v_0}(t)+4 {\eta}^2v_0^2 \\
\ \ \ \ \ \ = (y_{v_0}(t)-\sqrt{{\eta}^2+2}+3\eta)^2 \bigg(2\eta y + \frac{4{\eta}^2v_0^2 }{(\sqrt{{\eta}^2+2}-3\eta )^2} \bigg)
\end{aligned}
\end{equation*}

Recall Lemma~\ref{y_v}, $y_{v_0}(t)$ is given by the equation
$$t=\int_0^{y_{v_0}(t)} \frac{dz}{\sqrt{G(z)+v_0^2}}. $$
Combining the equality proved in \lemref{monotone} that $\int_0^{\sqrt{{\eta}^2+2}-3\eta} \frac{dz}{\sqrt{G(z)+v_0^2}}=+\infty$, we conclude that $\lim\limits_{t \to +\infty} y_{v_0}(t)=\sqrt{{\eta}^2+2}-3\eta.$ Also, according to \eqref{solution}, we get that $$\lim\limits_{t \to +\infty} \frac{d}{dt}y_{v_0}(t)=\sqrt{G(\sqrt{{\eta}^2+2}-3\eta)+v_0^2 }=0 .$$ 
Therefore by \eqref{x_v}, we conclude the second claim
$$\lim\limits_{t \to +\infty}x_{v_0}(t)=\frac{1}{2} (\sqrt{{\eta}^2+2}-3\eta)^2+2\eta (\sqrt{{\eta}^2+2}-3\eta)=1-{\eta}^2-\eta\sqrt{{\eta}^2+2}.$$

It can be seen that  the curves $\{x_v(t): t \geq 0, v \geq v_0 \}$ never cross each other, and that $x_v(t)< 1-{\eta}^2-\eta\sqrt{{\eta}^2+2}$ for any $ t>0$ if $v<v_0$. Therefore according to Lemma~\ref{10}, if $|2\bar{\theta}-1| \geq 1-{\eta}^2-\eta\sqrt{{\eta}^2+2}$, there exists only one $v \geq v_0$ such that $x_v(T)=2\bar{\theta}-1$. 

Now suppose that $0<2\bar{\theta}-1<1-{\eta}^2-\eta\sqrt{{\eta}^2+2}$. For each $v \in (0,v_0)$, define $$M(v):=\max\limits_{t \geq 0} x_v(t).$$
As a result of Lemma~\ref{10}, $M(v)$ is actually an increasing function, and there exists a unique $\bar{v} \in (0, v_0)$ such that $M(\bar{v})=2\bar{\theta}-1$. Also for any $v \in [\bar{v},v_0)$, we can define $t(v)$ as the unique $t$ satisfying $x_v(t)=2\bar{\theta}-1, t < T_1(v),$ which is also an increasing function of $v$. Then $(x_{v}(.), y_{v}(.))$ is a solution of \eqref{eq3} with time horizon $T=t(v)$. Since the period of $x_v(.)$ is $4T(v)$, and $\lim\limits_{v \to v_0}t(v)=+\infty$, for each $k \in \mathbb{N}$ we know that if $T>t(\bar{v})+4kT(\bar{v})$, there must 
exist some $v' \in [\bar{v}, v_0)$ such that $T=t(v')+4kT(v')$. Therefore we conclude that the number of solutions to \eqref{eq3} with time horizon $T$ is greater than 
$$\sup\limits_{k \in \mathbb{N}}\{k: T \geq t(\bar{v})+4kT(\bar{v})\},$$
which can be arbitrarily large if $T$ is large enough.

In the end, we consider the number of solutions for the terminal condition $2\bar{\theta}-1=0$. We have already shown that $T_k(v)$ is the time when $x_v(t)$ attains zero. According to \lemref{monotone}, the functions $T_k(v)$ are increasing with respect to $v$ for each $k \in \mathbb{N}$ and $\lim\limits_{v \to v_0} T_k(v)=+\infty$. Since $x_{-v}(t)=-x_v(t)$, and $v=0$ is always a solution, the number of solutions is just $$1+2 \{(k,v): T_k(v)=T, k \in \mathbb{N}, v \in (0, v_0)      \}=1+2 \sup\limits_{k \in \mathbb{N}} \{k : T_k(0+) <T \}.$$
\end{proof}

\section{The Master Equation}\label{eq:ME}

Letting $Y(t,\theta)=U(t,1,\theta)-U(t,0,\theta)$, $x=2\theta-1$, and time reverse the master equation \eqref{ME},  we obtain the equation 
\begin{equation}\label{entropy}
\begin{aligned}
\frac{\partial Y}{\partial t}+ \frac{\partial }{\partial x}\bigg(2\eta xY +\frac{xY|Y|}{2}-\frac{Y^2}{2}-\frac{x^2}{2}\bigg)=0,
\end{aligned}
\end{equation}
with the boundary condition $Y(0,x)=0, \forall x \in [-1,1]$. 

Since the equation has the form of a scalar conservation law, there exists a unique entropy solution. By the method of characteristics, we directly construct a piecewise $\mathcal{C}^1$ solution to \eqref{entropy} and then check it is entropic.

Rewriting \eqref{entropy} as 
$$\frac{\partial{Y}}{\partial{t}} +\frac{\partial{Y}}{\partial{x}}(2\eta x-Y+x|Y|)=-2\eta Y -\frac{Y|Y|}{2}+x, $$
and letting $y(t)=Y(t,x(t)), \frac{d}{dt} x=2 \eta x  - y+x|y|$, we obtain the 
characteristic curve of \eqref{entropy} 
\[
\begin{cases}
\frac{d}{dt} x=2 \eta x  - y+x|y|, \\
\frac{d}{dt} y=-2\eta y-\frac{y|y|}{2}+x, \\
y(0)=0, x(0)=\frac{dy}{dt}(0),
\end{cases}
\]
whose solution is given explicitly in Lemma~\ref{y_v}. If $\eta \geq \frac{1}{2}$, the solution given by characteristic curves is smooth everywhere. If $\eta< \frac{1}{2}$, the shock curve is taken to be $\gamma(t)=0, t \in \mathbb{R}_+.$ See our illustration in \figref{curve}.

\begin{figure}[h]
\caption{Characteristic curves, $\eta=0.1, T=3$ on the left; $\eta=0.6, T=1$ on the right.}
\label{curve}
\includegraphics[width=2in]{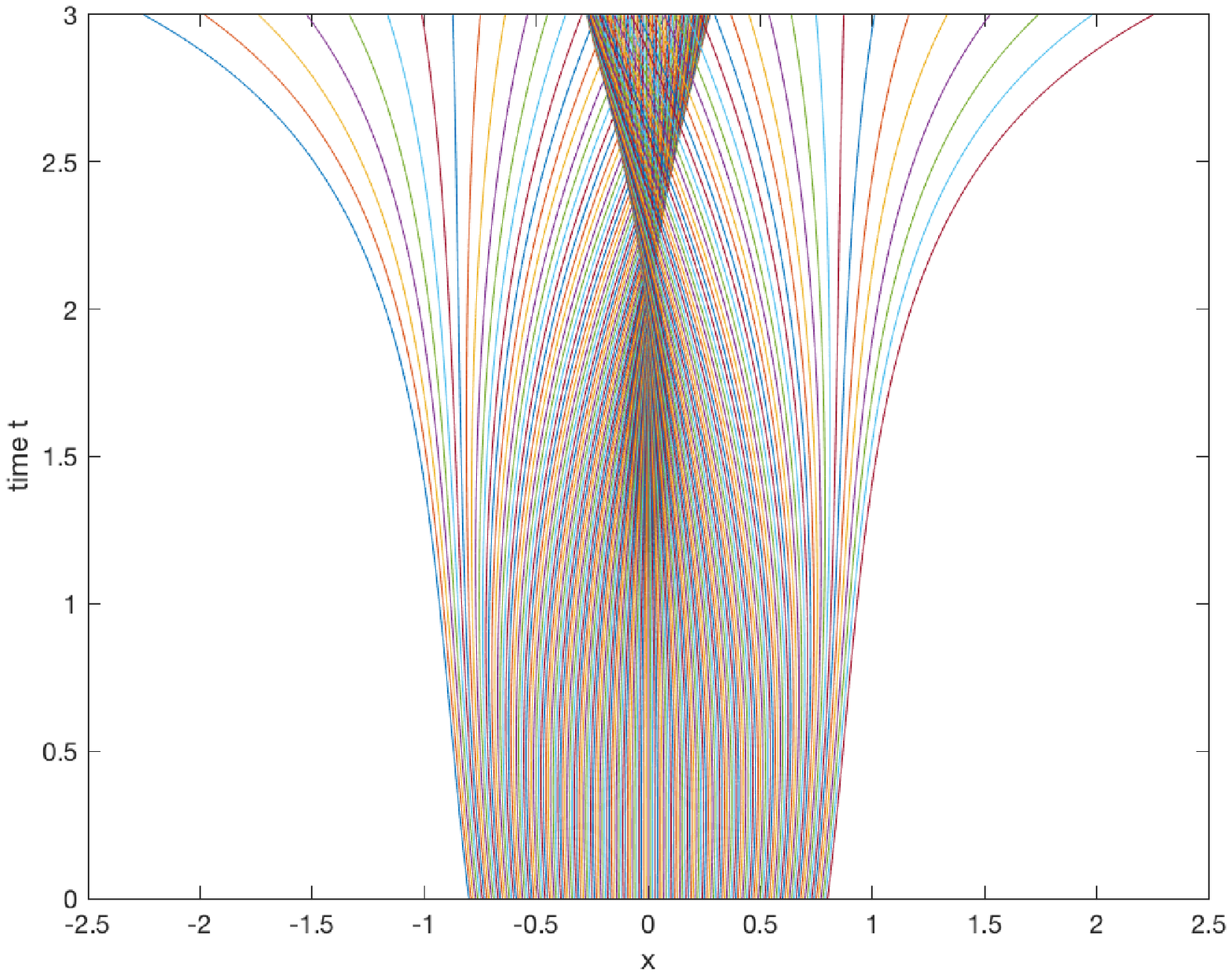}\includegraphics[width=2in]{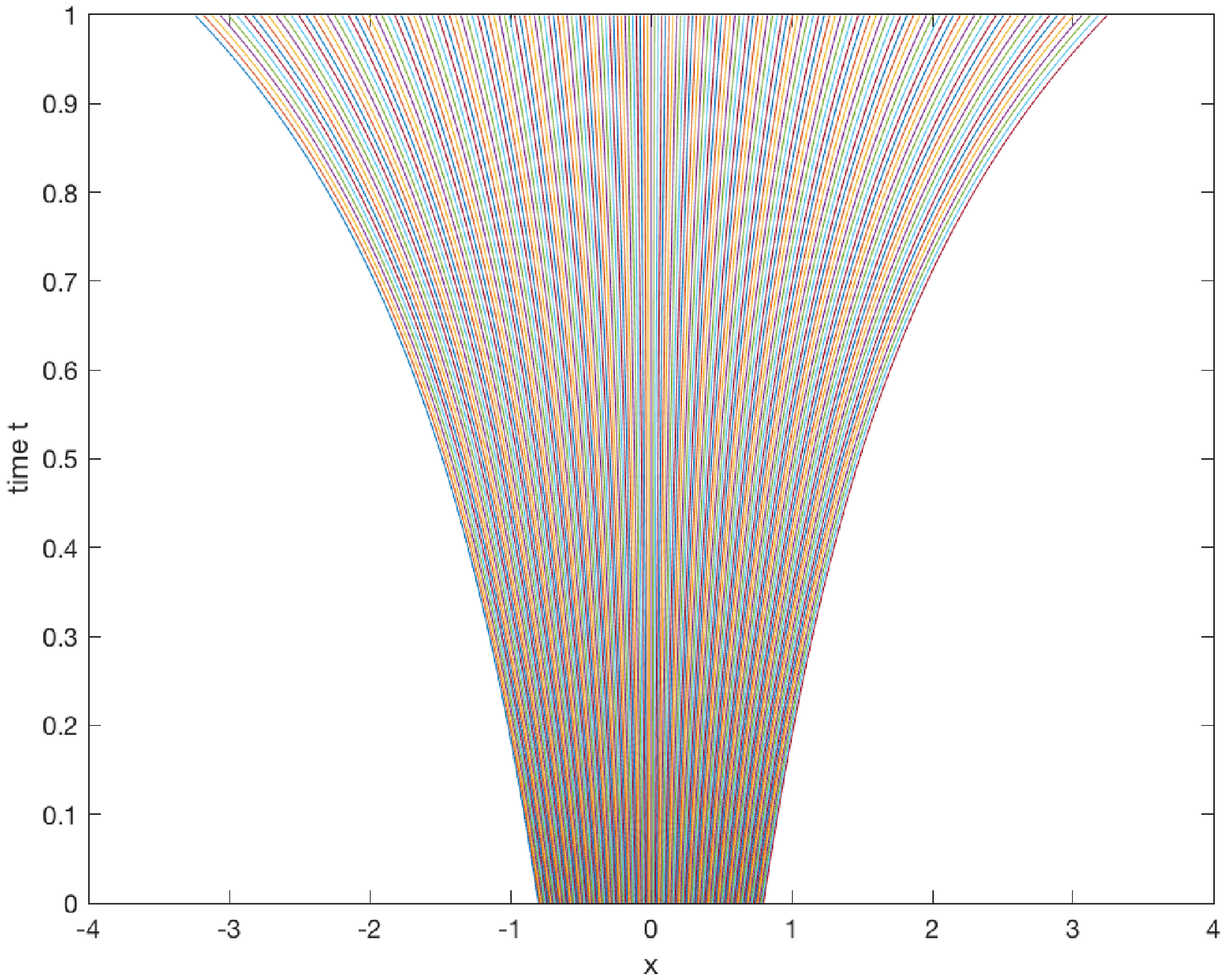}
\end{figure}

\begin{prop}
The function $Y(x,t):=y_{v(x,t)}(t)$ is the entropy solution of \eqref{entropy} with shock curve $\gamma(t)=0, t > T_1(0+)$, where $v(x,t) \in \mathbb{R}$ is defined in Lemma~\ref{10}. 
\end{prop}
\begin{proof}
It is clear that the function $Y(x,t)$ is $\mathcal{C}^1$ outside the shock curve, and we only need to check the $Rankine$-$Hugoniot \ condition$ and the $Lax \ condition$ (see \cite[Proposition 3]{MR3981375}). Define $$Y_+(t):=\lim\limits_{x \downarrow 0} Y(x,t), \ \ Y_-:=\lim\limits_{x \uparrow 0} Y(x,t). $$

If $t > T_1(0+)$, there exists a $v>0$ such that $t=T_1(v)$ since $v \mapsto T_1(v)$ is increasing to $+\infty$ as $v$ increases to $v_0$. Also it can be seen that $\lim\limits_{x \downarrow 0} v(x,t)=v$. According to the discussion above Lemma~\ref{10}, we conclude that $Y_+(t)=y_v(t)=v=\lim\limits_{x \downarrow 0} v(x,t)$,  
and similarly $Y_-(t)=-\lim\limits_{x \downarrow 0} v(x,t)$. If $t \leq T_1(0+)$, the mapping $v \mapsto x_v(t)$ is continuous and strictly increasing, which is zero at  $v=0$. Therefore $\lim\limits_{x \downarrow 0} v(x,t)=0$, and $Y_+(t)=Y_-(t)=0$. In summary, we have 

\[
Y_+(t)=-Y_-(t)=
\begin{cases}
\lim\limits_{x \downarrow 0}v(x,t) &  \text{ if } t > T_1(0+), \\
0                                                  & \text{ if } t \leq T_1(0+). 
\end{cases}
\]

Taking $\mathfrak{g}(x,Y)=2\eta xY +\frac{xY|Y|}{2}-\frac{Y^2}{2}-\frac{x^2}{2}$, 
we have $$\frac{d}{dt} \gamma(t)=0=  \frac{-\frac{(Y_+(t))^2}{2}+\frac{(Y_-(t))^2}{2}}{Y_+(t)-Y_-(t) }=\frac{\mathfrak{g}(\gamma(t),Y_+(t))-\mathfrak{g}(\gamma(t), Y_-(t))}{Y_+(t)-Y_-(t) },$$
which verifies the $Rankine$-$Hugoniot \ condition$. 

For any $c$ strictly between $Y_-(t)$ and $Y_+(t)$, $t >T_1(0+)$, we have $$\frac{\mathfrak{g}(\gamma(t),c)-\mathfrak{g}(\gamma(t),Y_+(t)) }{c-Y_+(t)}=\frac{\frac{(Y_+(t))^2}{2}-\frac{c^2}{2}}{c-Y_+(t)}=-\frac{c+Y_+(t)}{2},$$
$$\frac{\mathfrak{g}(\gamma(t),c)-\mathfrak{g}(\gamma(t),Y_-(t)) }{c-Y_-(t)}= =\frac{\frac{(Y_-(t))^2}{2}-\frac{c^2}{2}}{c-Y_-(t)}=-\frac{c+Y_-(t)}{2},$$
and therefore
$$\frac{\mathfrak{g}(\gamma(t),c)-\mathfrak{g}(\gamma(t),Y_+(t)) }{c-Y_+(t)}< \frac{d}{dt} \gamma(t)=0 <\frac{\mathfrak{g}(\gamma(t),c)-\mathfrak{g}(\gamma(t),Y_-(t)) }{c-Y_-(t)} ,$$
which verifies the $Lax \ condition$.
\end{proof}

\begin{remark}
It is easily seen that the entropy solution of \eqref{entropy} corresponds to a solution of \eqref{ME}.
\end{remark}

\begin{remark}\label{rmkMFG}
By Lemma~\ref{10}, we know that for any $\bar{\theta} \in [0,1]$, there exists a unique $v^{'}$ such that $x_{v^{'}}(T)=2 \bar{\theta}-1, T<T_1(v^{'})$. Then $(x_{v^{'}}(T-t), y_{v^{'}}(T-t))$ solves \eqref{eq1}, which is the mean field equilibrium induced the entropy solution. 
\end{remark}

\section{$N+1$-player game and the selection of Equilibrium}\label{sec:conv}
In this section, we consider the $N+1$-player game and always assume $\eta=0$. Since the model we are considering is invariant under permutation, it can be easily seen that $$V^{N+1}(t, 0, 1-\theta)=V^{N+1}(t, 1, \theta),$$ and therefore we only need to consider the HJB systems for $V^{N+1}(t,1,\theta)$: 
\begin{equation}\label{HJB}
\begin{cases}
 -\frac{d}{dt} V^{N+1}(t,1,\theta) = f(1,\theta)- \frac{(\alpha^{N+1}_*(t,1,\theta))^2 }{2}\\
 \ \ \ \ \ \ \ \ \ \ \ \ \ \ \ \ \ \ \  +N(1-\theta)\alpha^{N+1}_*(t,1,\theta)(V^{N+1}(t,1, \theta +\frac{1}{N})-V^{N+1}(t,1,\theta))\\
 \ \ \ \ \ \ \ \ \ \ \ \ \ \ \ \ \ \ \  +N\theta \alpha^{N+1}_*(t,0,\theta-\frac{1}{N})(V^{N+1}(t,1,\theta-\frac{1}{N})-V^{N+1}(t,1,\theta)) \\
 V^{N+1}(T,1,\theta)= 0, \\
\end{cases}
\end{equation}
where the optimal control policy is 
$$a^{N+1}_*(t, i, \theta)=(V^{N+1}(t,i,\theta)-V^{N+1}(t, 1-i, \theta))_+.$$

As a result of the local Lipschitz continuity of the HJB equation \eqref{HJB}, the system can be uniquely solved with terminal condition $V^{N+1}(T, 0, \theta)=0$, which provides us the unique Nash Equilibrium of the game.  Supposing that the representative player is applying the zero control while the other players are taking the optimal policy, then by the definition of Nash Equilibrium we conclude that $$V^{N+1}(t,1, \theta) \leq \mathbb{E} \bigg[ \int_t^T f(i(t), \theta_t) dt \bigg] \leq T-t. $$
Now we prove that  if the representative player agrees with the majority, then he will keep his state by taking the zero control. 

\begin{prop}\label{main}
Taking $$Y^{N+1}(t, \theta)=V^{N+1}(t,1, \theta)-V^{N+1}(t, 0, \theta)=V^{N+1}(t,1, \theta)-V^{N+1}(t, 1, 1- \theta),$$
for any $\theta \in \{0, \frac{1}{N}, \dotso, 1        \}$ we have 
\begin{equation}\label{majority}
\begin{aligned}
Y^{N+1}(t, \theta) \geq 0 \ (\alpha^{N+1}_*(t,0,\theta)=0)  \ \ \ \ &  \text {if } \theta \geq \frac{1}{2} ,  \\
Y^{N+1}(t, \theta) \leq 0 \ (\alpha^{N+1}_*(t, 1, \theta)=0) \ \ \ \ & \text{if } \theta \leq \frac{1}{2}. \\
\end{aligned}
\end{equation}
\end{prop}
\begin{proof}
We only prove the first inequality of \eqref{majority} for even $N$, and the rest can be proved similarly. As a result of $Y^{N+1}(t, \frac{1}{2})=0$, it is enough for us to show it for $\theta \geq \frac{1}{2}+\frac{1}{N}$. Take 
$$W^{N+1}(t, \theta)=V^{N+1}(t,1,\theta)-V^{N+1}(t,1, \theta-\frac{1}{N}).$$
According to \eqref{HJB}, we obtain 
\begin{equation}\label{value1}
\begin{aligned}
\frac{d}{dt} Y^{N+1}& (t, \theta)=   1-2 \theta + \frac{|Y^{N+1}(t, \theta)| Y^{N+1}(t, \theta) }{2}   \\
  & + N \theta \bigg( Y^{N+1}(t, \theta -\frac{1}{N})_-  W^{N+1}(t, \theta ) + Y^{N+1}(t, \theta)_-W^{N+1}(t, 1- \theta+\frac{1}{N}) \bigg) \\
&  -N(1-\theta) \bigg(Y^{N+1}(t, \theta)_+W^{N+1}(t ,\theta+\frac{1}{N}) +Y^{N+1}(t,\theta+\frac{1}{N})_+W^{N+1}(t, 1- \theta)         \bigg), \\
 \end{aligned}
\end{equation}
and 
\begin{equation}\label{value2}
\begin{aligned}
\frac{d}{dt} W^{N+1}(t, 1-\theta)= & -\frac{1}{N}+\frac{Y^{N+1}(t, 1-\theta)_+^2 }{2}-\frac{Y^{N+1}(t, 1-\theta-\frac{1}{N})_+^2 }{2}      \\
&-N\theta Y^{N+1}(t, 1-\theta)_+W^{N+1}(t ,1-\theta+\frac{1}{N})  \\
& +N(1-\theta) Y^{N+1}(t, 1-\theta-\frac{1}{N})_-W^{N+1}(t,1- \theta )    \\
&+N(\theta+\frac{1}{N})Y^{N+1}(t, 1-\theta -\frac{1}{N})_+W^{N+1}(t, 1-\theta) \\
&-N(1-\theta-\frac{1}{N})Y^{N+1}(t, 1-\theta -\frac{2}{N})_-W^{N+1}(t, 1- \theta -\frac{1}{N}). \\
\end{aligned}
\end{equation}

By our terminal condition $V^{N+1}(T,1, \theta) =0$, it is easy to see that $Y^{N+1}(T,\theta)=W^{N+1}(T,\theta)=0$, and both $\frac{d}{dt} Y^{N+1}(T, \theta ), \frac{d}{dt} W^{N+1}(T, 1- \theta)$ are negative if $\theta> \frac{1}{2}.$ And therefore by the continuity of $V^{N+1}(t,1, \theta)$, there exists a small positive $\epsilon >0$ such that $Y^{N+1}(t,\theta), W^{N+1}(t, 1-\theta)$ are positive during the time interval $[T- \epsilon,T)$. Define $$s:=\sup\limits_{\{t<T-\epsilon\}} \{ t: W^{N+1}(t,1-\theta)=0   \text { or } Y^{N+1}(t, \theta)=0 \text{ for some } \theta >\frac{1}{2}         \} .$$
We finish the argument by showing that $Y^{N+1}(t, \theta)$ and $W^{N+1}(t, 1-\theta)$ are both positive for $t \in [s, T-\epsilon], \theta > \frac{1}{2}$, which implies $s$ has to be $-\infty$. 
By the definition of $s$, we have $Y^{N+1}(t, \theta)=-Y^{N+1}(t,1-\theta) \geq 0, W^{N+1}(t, 1- \theta ) \geq 0 $ if $t \in [s, T-\epsilon)$, $\theta > \frac{1}{2}$, and therefore we obtain the following inequality from \eqref{value1},
$$\frac{d}{dt} Y^{N+1}(t, \theta ) \leq Y^{N+1}(t, \theta) \bigg(\frac{Y^{N+1}(t, \theta)}{2}- N(1- \theta) W^{N+1}(t, \theta +\frac{1}{N})\bigg).$$ Since $V^{N+1}(t,1, \theta ) \leq T$, we get that $|Y^{N+1}(t, \theta)| \leq 2T$, $|W^{N+1}(t, \theta)| \leq 2T$ for any $\theta \in \{0, \frac{1}{N}, \dotso,1 \}$.  Therefore $Y^{N+1}(t, \theta)$ is bounded below by the solution of 
\begin{equation*}
\begin{cases}
\frac{d}{dt} l(t)= ( T+2NT )l(t)\\
l(T-\epsilon)=Y^{N+1}(T-\epsilon, \theta), \\
\end{cases}
\end{equation*}
which is always positive. Similarly, for $t \in [s, T-\epsilon], \theta > \frac{1}{2}$, we obtain the inequality from \eqref{value2} 
$$ \frac{d}{dt} W^{N+1}(t, 1- \theta) \leq N(1-\theta)Y^{N+1}(t, 1- \theta -\frac{1}{N})_-W^{N+1}(t, 1- \theta) \leq 2NT(1-\theta) W^{N+1}(t, 1- \theta) ,$$
which implies $W^{N+1}(t, 1- \theta) >0$ for $t \in [s, T-\epsilon]$. 
\end{proof}

\begin{remark}
Recall that $\mathbf{Z}(t)$ is the state of the $N+1$ players at time $t$  when agents play the Nash equilibrium given by \eqref{HJB1}. Denote by ${\theta}^{N+1}(t)$ the fraction of players at state $0$, i.e., $${\theta}^{N+1}(t)=\frac{1}{N+1} \sum\limits_{j=1}^{N+1} \delta_{Z_j(t)=0}. $$
and let $U$ be the solution of \eqref{ME} corresponding to the entropy solution of \eqref{entropy}.  According to Proposition~\ref{main}, ${\theta}^{N+1}(t)$ will always stay on one side of $\frac{1}{2}$ if ${\theta}^{N+1}(0) \not =\frac{1}{2}$.  In combination with the fact that $U(t,i, \theta)$ is smooth outside the curve $\bar{\gamma}(t)=\frac{1}{2}$, it can be easily seen that $V^{N+1}(t,1,\theta)$ converges to $U(t,1,\theta)$ if $\theta \not =\frac{1}{2}$ (see  e.g. \cite[Theorem 8]{MR3981375}        ).

Let $(\xi_j)_{j \in \mathbb{N}}$ be  the i.i.d initial datum of $Z_j$ such that  $\mathbb{P}[\xi_j=0]=\bar{\theta}\not = \frac{1}{2}, \mathbb{P}[\xi_j=1]=1-\bar{\theta}.$ 
Denote by $\tilde{Z}_j $ the i.i.d process in which players choose the optimal control $\tilde{\alpha}(t,i):=(U(t,i,\theta(t))-U(t,1-i,\theta(t)))_+$, where $U$ is the corresponding entropy solution of \eqref{ME}. Also, we can prove the propagation of chaos property by using the technique developed in \cite{MR4013871} and \cite{MR3981375}. 
\end{remark}

\section{Conclusion}When $\eta>1/2$, the N-player game converges to the mean field game following the analysis of \cite{MR3860894} and \cite{MR4013871}. Here we considered the case when $\eta=0$ and showed that the N-player game value functions converge to the entropic mean-field game solution and verified in this case the conjecture of \cite{MR4046528}.

When $\eta \in (0,\frac{1}{2})$, it is always possible for players to jump to the other state. Therefore $\theta^{N+1}(t)$ may not always stay on one side of $\frac{1}{2}$, and when we use It\^{o}'s formula to the entropy solution $U$, there would be extra jump terms. Subsequently our strategy does not work when $\eta \in (0,1/2)$, and  new techniques are needed. We leave this as an open problem.

When $\bar{\theta}=1/2$, it is expected that the N player limit will charge the two solutions we obtain with equal probability (as in \cite{MR4046528}), which is numerically justified by the Figure 3 of \cite{2019arXiv190305788H}. Hence in that case the $N$-player empirical distribution will not converge to the stable fixed points of the MFG map (in the language of \cite{MR4046528}) unlike what is claimed in the conjecture.

\footnotesize{
\bibliographystyle{siam}
\bibliography{ref.bib}}

\begin{thebibliography}{10}

\bibitem{MR3860894}
{\sc E.~Bayraktar and A.~Cohen}, {\em Analysis of a finite state many player
  game using its master equation}, SIAM J. Control Optim., 56 (2018),
  pp.~3538--3568.

\bibitem{MR3967062}
{\sc P.~Cardaliaguet, F.~Delarue, J.-M. Lasry, and P.-L. Lions}, {\em The
  master equation and the convergence problem in mean field games}, vol.~201 of
  Annals of Mathematics Studies, Princeton University Press, Princeton, NJ,
  2019.

\bibitem{MR3752669}
{\sc R.~Carmona and F.~Delarue}, {\em Probabilistic theory of mean field games
  with applications. {I}}, vol.~83 of Probability Theory and Stochastic
  Modelling, Springer, Cham, 2018.
\newblock Mean field FBSDEs, control, and games.

\bibitem{MR3753660}
\leavevmode\vrule height 2pt depth -1.6pt width 23pt, {\em Probabilistic theory
  of mean field games with applications. {II}}, vol.~84 of Probability Theory
  and Stochastic Modelling, Springer, Cham, 2018.
\newblock Mean field games with common noise and master equations.

\bibitem{MR4013871}
{\sc A.~Cecchin and G.~Pelino}, {\em Convergence, fluctuations and large
  deviations for finite state mean field games via the master equation},
  Stochastic Process. Appl., 129 (2019), pp.~4510--4555.

\bibitem{MR3981375}
{\sc A.~Cecchin, P.~D. Pra, M.~Fischer, and G.~Pelino}, {\em On the
  {C}onvergence {P}roblem in {M}ean {F}ield {G}ames: {A} {T}wo {S}tate {M}odel
  without {U}niqueness}, SIAM J. Control Optim., 57 (2019), pp.~2443--2466.

\bibitem{MR4046528}
{\sc F.~Delarue and R.~Foguen~Tchuendom}, {\em Selection of equilibria in a
  linear quadratic mean-field game}, Stochastic Process. Appl., 130 (2020),
  pp.~1000--1040.

\bibitem{MR3268060}
{\sc D.~Gomes, R.~M. Velho, and M.-T. Wolfram}, {\em Socio-economic
  applications of finite state mean field games}, Philos. Trans. R. Soc. Lond.
  Ser. A Math. Phys. Eng. Sci., 372 (2014), pp.~20130405, 18.

\bibitem{MR3072242}
{\sc D.~A. Gomes, J.~Mohr, and R.~R.~a. Souza}, {\em Continuous time finite
  state mean field games}, Appl. Math. Optim., 68 (2013), pp.~99--143.

\bibitem{2019arXiv190305788H}
{\sc B.~{Hajek} and M.~{Livesay}}, {\em {On non-unique solutions in mean field
  games}}, arXiv e-prints,  (2019), p.~arXiv:1903.05788.

\bibitem{4303232}
{\sc M.~{Huang}, P.~E. {Caines}, and R.~P. {Malhame}}, {\em Large-population
  cost-coupled lqg problems with nonuniform agents: Individual-mass behavior
  and decentralized $\varepsilon$-nash equilibria}, IEEE Transactions on
  Automatic Control, 52 (2007), pp.~1560--1571.

\bibitem{MR2346927}
{\sc M.~Huang, R.~P. Malham\'{e}, and P.~E. Caines}, {\em Large population
  stochastic dynamic games: closed-loop {M}c{K}ean-{V}lasov systems and the
  {N}ash certainty equivalence principle}, Commun. Inf. Syst., 6 (2006),
  pp.~221--251.

\bibitem{MR2269875}
{\sc J.-M. Lasry and P.-L. Lions}, {\em Jeux \`a champ moyen. {I}. {L}e cas
  stationnaire}, C. R. Math. Acad. Sci. Paris, 343 (2006), pp.~619--625.

\bibitem{LASRY2006679}
{\sc J.-M. Lasry and P.-L. Lions}, {\em Jeux {\`a} champ moyen. ii -- horizon
  fini et contr{\^o}le optimal}, Comptes Rendus Mathematique, 343 (2006),
  pp.~679 -- 684.

\bibitem{MR2295621}
{\sc J.-M. Lasry and P.-L. Lions}, {\em Mean field games}, Jpn. J. Math., 2
  (2007), pp.~229--260.

\end{thebibliography}
\end{document}